\documentclass[a4paper]{amsart}
\usepackage{amssymb}
\usepackage{ifthen}
\usepackage{graphicx}

\newtheorem{thm}{Theorem}
\newtheorem{cor}{Corollary}
\newtheorem{lem}{Lemma}

\newtheorem{rem}{Remark}

\newtheorem{conj}{Conjecture}
\theoremstyle{definition}
\newtheorem{example}[equation]{Example}
\newtheorem{prob}[equation]{Problem}

\newcommand{\U}{{\mathcal U}}

\newcommand{\CC}{{\mathcal C}}

\newcommand{\D}{{\mathbb D}}

\def\be{\begin{equation}}
\def\ee{\end{equation}}

\newcommand{\bee}{\begin{enumerate}}
\newcommand{\eee}{\end{enumerate}}

\newcommand{\blem}{\begin{lem}}
\newcommand{\elem}{\end{lem}}
\newcommand{\bthm}{\begin{thm}}
\newcommand{\ethm}{\end{thm}}
\newcommand{\bcor}{\begin{cor}}
\newcommand{\ecor}{\end{cor}}
\newcommand{\beg}{\begin{example}}
\newcommand{\eeg}{\end{example}}
\newcommand{\begs}{\begin{examples}}
\newcommand{\eegs}{\end{examples}}
\newcommand{\bdefe}{\begin{defin}}
\newcommand{\edefe}{\end{defin}}
\newcommand{\bprob}{\begin{prob}}
\newcommand{\eprob}{\end{prob}}
\newcommand{\bei}{\begin{itemize}}
\newcommand{\eei}{\end{itemize}}

\newcommand{\bcon}{\begin{conj}}
\newcommand{\econ}{\end{conj}}
\newcommand{\bcons}{\begin{conjs}}
\newcommand{\econs}{\end{conjs}}
\newcommand{\bprop}{\begin{propo}}
\newcommand{\eprop}{\end{propo}}
\newcommand{\br}{\begin{rem}}
\newcommand{\er}{\end{rem}}
\newcommand{\brs}{\begin{rems}}
\newcommand{\ers}{\end{rems}}
\newcommand{\bo}{\begin{obser}}
\newcommand{\eo}{\end{obser}}
\newcommand{\bos}{\begin{obsers}}
\newcommand{\eos}{\end{obsers}}
\newcommand{\bpf}{\begin{pf}}
\newcommand{\epf}{\end{pf}}
\newcommand{\ba}{\begin{array}}
\newcommand{\ea}{\end{array}}
\newcommand{\beq}{\begin{eqnarray}}
\newcommand{\beqq}{\begin{eqnarray*}}
\newcommand{\eeq}{\end{eqnarray}}
\newcommand{\eeqq}{\end{eqnarray*}}

\begin{document}
\bibliographystyle{amsplain}

\title[On a special class of Schwartz functions]{On a special class of Schwartz functions}

\author[M. Obradovi\'{c}]{Milutin Obradovi\'{c}}
\address{Department of Mathematics,
Faculty of Civil Engineering, University of Belgrade,
Bulevar Kralja Aleksandra 73, 11000, Belgrade, Serbia.}
\email{obrad@grf.bg.ac.rs}

\author[N. Tuneski]{Nikola Tuneski}
\address{Department of Mathematics and Informatics, Faculty of Mechanical Engineering, Ss. Cyril and
Methodius
University in Skopje, Karpo\v{s} II b.b., 1000 Skopje, Republic of North Macedonia.}
\email{nikola.tuneski@mf.edu.mk}


\begin{abstract}
In this paper we study functions $  \omega(z) = c_1z+c_2z^2+c_3z^3+\cdots$ analytic in the open unit disk $\D$ and such that $|\omega'(z)|\le1$ for all $z\in\D$. For these functions we give estimates (sometimes sharp) for the following moduli: $|c_3-c_1c_2|$, $|c_1c_3-c_2^2|$, and $|c_4-c_2^2|$.
\end{abstract}

\maketitle

\medskip

\section{Introduction and definitions}

For a function $\omega$, analytic in the open unit disk $\D = \{z:|z|<1\}$ and of the form
\begin{equation}\label{e1}
  \omega(z) = c_1z+c_2z^2+c_3z^3+\cdots,\qquad (c_1,c_2,\ldots\in\CC)
\end{equation}
we say that is Schwartz function if $|\omega(z)|<1$, $z\in\D$. We denote by $\mathcal{B}_0$ the class of all such functions.

\medskip

In his paper \cite{zaprawa}, Zaprawa gave many different inequalities for the coefficients $c_1, c_2,\ldots$ for the functions of the class $\mathcal{B}_0$.

\medskip

In this paper we study the class of functions $\mathcal{B}'_0$ of type \eqref{e1} such that $|\omega'(z)|\le 1$ for all $z\in\D$. Since
\begin{equation}\label{e2}
  z\omega'(z) = c_1z+2c_2z^2+3c_3z^3\cdots,
\end{equation}
and $|z\omega'(z)| = |z|\cdot |\omega'(z)|\le|z|<1$, $z\in\D$, it means that $z\omega'(z)$ belongs to $\mathcal{B}_0$. Also, since $\omega(z) = \int_0^z\omega'(z)dz$, then $|\omega(z)| \le \int_0^{|z|} |\omega'(z)|\,d|z| \le |z|<1$ for all $z\in\D$, i.e., $\omega\in \mathcal{B}_0$. So, $|\omega'(z)|\le 1$, $z\in\D$, is a sufficient condition for $\omega\in \mathcal{B}_0$, i.e., $\mathcal{B}'_0$ is subclass of the class $\mathcal{B}_0$.

\medskip

For the functions from $\mathcal{B}'_0$ we try to find properties for the coefficients $c_1,c_2,c_3,\ldots$ that correspond to the properties for the functions from $\mathcal{B}_0$.

\medskip

For our considerations we will need the next lemma originating from \cite{carlson}.

\begin{lem}\label{lem-carl}
Let $\omega\in \mathcal{B}_0$ is given by \eqref{e1}. Then
\begin{equation}\label{e3}
\begin{array}{c}
|c_1|\le1,\qquad  |c_2|\le1-|c_1|^2,\\[4mm]
|c_3|\le 1-|c_1|^2-\frac{|c_2|^2}{1+|c_1|},\\[4mm]
|c_4|\le1-|c_1|^2 -|c_2|^2.
\end{array}
\end{equation}
\end{lem}

We showed that when $\omega$ given by \eqref{e1} is in  $\mathcal{B}_0$, then $z \omega'(z)$ is in $\mathcal{B}'_0$. Thus, Lemma \ref{lem-carl}, together with \eqref{e2}, directly brings

\begin{lem}\label{lem-B}
Let $\omega\in \mathcal{B}'_0$ is given by \eqref{e1}. Then
\begin{equation}\label{e4}
\begin{split}
|c_1|&\le1,\qquad  |c_2|\le\frac12\left(1-|c_1|^2\right), \\[4mm]
|c_3|&\le \frac13\left(1-|c_1|^2-\frac{4|c_2|^2}{1+|c_1|}\right),\\[4mm]
 |c_4|&\le\frac14\left(1-|c_1|^2 -4|c_2|^2\right).
\end{split}
\end{equation}
\end{lem}

\medskip

\section{Main results}

We begin with partly sharp estimate of the modulus $|c_3-c_1c_2|$ for functions from $\mathcal{B}'_0$ with expansion \eqref{e1}.

\begin{thm}\label{th-1}
If $\omega\in \mathcal{B}'_0$ is of form \eqref{e1}, then
\begin{equation}\label{e5}
  |c_3-c_1c_2|\le
  \left\{
  \begin{array}{cc}
    \frac{1}{48}(1+|c_1|)\left[ 9|c_1|^2-16|c_1|+16 \right], & 0\le|c_1|\le\frac47 \\[2mm]
    \frac56|c_1|(1-|c_1|^2), & \frac47\le|c_1|\le1
  \end{array}
  \right..
\end{equation}
The estimate is sharp for $|c_1|=0$ and for $\frac47\le|c_1|\le1$.
\end{thm}

\begin{proof}
For $\omega\in \mathcal{B}'_0$ and $\omega$ given by \eqref{e1} we apply the inequalities \eqref{e4}:
\[
\begin{split}
|c_3-c_1c_2| &\le |c_3|+|c_1||c_2| \le \frac13\left( 1-|c_1|^2-\frac{4|c_2|^2}{1+|c_1|} \right) +|c_2||c_2|\\
& =-\frac{4}{3(1+|c_1|)}|c_2|^2 +|c_1||c_2| +\frac13(1-|c_1|^2).
\end{split}
\]
If we consider the last expression as a function of $|c_2|$, $0\le|c_2|\le\frac12(1-|c_1|^2)$, then we easily obtain the estimate given by \eqref{e5}, depending on its maximum  which in the case $0\le|c_1|\le\frac47$ is attained for  $|c_2| = \frac38|c_1|(1+|c_1|)$ lying in the  interval $\left(0,\frac12(1-|c_1|^2)\right)$, and in the case $\frac47\le|c_1|\le1$ is attained for  $|c_2| = \frac12(1-|c_1|^2)$.

\medskip

For $|c_1|=0$ and for $\frac47\le|c_1|\le1$ the result is sharp with extremal functions $\omega_1(z)=\frac13z^3$ and
\[ \omega_2(z) = \int_0^z \frac{|c_1|+z}{1+|c_1|z} \,dz = |c_1|z+\frac12(1-|c_1|^2)z^2-\frac13|c_1|(1-|c_1|^2)z^3+\cdots, \]
respectively.
\end{proof}

\medskip

\begin{rem}
Theorem \ref{th-1} brings:
\[ \omega \in \mathcal{B}'_0 \qquad \Rightarrow \qquad |c_3-c_1c_2| \le \frac13, \]
while
\[ \omega \in \mathcal{B}_0 \qquad \Rightarrow \qquad |c_3-c_1c_2| \le 1\]
follows from \cite{zaprawa}.
\end{rem}

\medskip

Similarly as Theorem \ref{th-1} we can prove the next theorem.

\begin{thm}\label{th-2}
If $\omega\in \mathcal{B}'_0$ is of form \eqref{e1} and $\mu\in\CC$, then
\begin{equation}\label{e6}
  |c_3-\mu c_1c_2|\le
  \left\{
  \begin{array}{cc}
    \frac{1}{48}(1+|c_1|)\left[ 9|\mu|^2|c_1|^2-16|c_1|+16 \right], & 0\le|c_1|\le\frac{1}{1+3/4|\mu|} \\[2mm]
    \left( \frac13+\frac12 |\mu| \right) |c_1|(1-|c_1|^2), & \frac{1}{1+3/4|\mu|}\le|c_1|\le1
  \end{array}
  \right..
\end{equation}
The estimate is sharp for $|c_1|=0$, and for $\frac{1}{1+3/4|\mu|}\le|c_1|\le1$ when $\mu$ is nonnegative real number. The  extremal functions are $\omega_1$ and $\omega_2$, respectively ($\omega_1$ and $\omega_2$ as defined in the proof of Theorem \ref{th-1}).
\end{thm}

\medskip

For $\mu=2$ in Theorem \ref{th-2} we receive

\begin{cor}\label{cor-2}
If $\omega\in \mathcal{B}'_0$ is of form \eqref{e1}. Then
\begin{equation*}
  |c_3-2c_1c_2|\le
  \left\{
  \begin{array}{cc}
    \frac{1}{12}(1+|c_1|)\left[ 9|c_1|^2-4|c_1|+4 \right], & 0\le|c_1|\le\frac{2}{5} \\[2mm]
    \frac43|c_1|(1-|c_1|^2), & \frac{2}{5}\le|c_1|\le1
  \end{array}
  \right..
\end{equation*}
The estimate is sharp for $|c_1|=0$ and for $\frac25\le|c_1|\le1$, with extremal functions $\omega_1$ and $\omega_2$, respectively ($\omega_1$ and $\omega_2$ as defined in the proof of Theorem \ref{th-1}).
\end{cor}

\medskip

Next, for the modulus $|c_1c_3-c_2^2|$ we have the following sharp estimate.

\medskip

\begin{thm}\label{th-3}
If $\omega\in \mathcal{B}'_0$ is of form \eqref{e1}, then the following estimate is sharp
\begin{equation}\label{e7}
  |c_1c_3-c_2^2|\le     \frac{1}{42}(1-|c_1|^2)(3+|c_1|^2), \quad  0\le|c_1|\le1.
\end{equation}
\end{thm}

\begin{proof}
Using Lemma \ref{lem-B} we have
\[
\begin{split}
|c_1c_3-c_2^2| &\le |c_1||c_3|+|c_2|^2 \\
&\le |c_1|\cdot \frac13\left( 1-|c_1|^2-\frac{4|c_2|^2}{1+|c_1|} \right)+|c_2|^2\\
 &=\frac13|c_1|(1-|c_1|^2)+ |c_2|^2 \cdot \frac{3-|c_1|}{3(1+|c_1|)}\\
 &\le \frac13|c_1|(1-|c_1|^2)+\frac{3-|c_1|}{3(1+|c_1|)}\cdot \frac14(1-|c_1|^2)^2\\
&= \frac{1}{12}(1-|c_1|^2)(2+|c_1|^2).
\end{split}
\]
The equality in \eqref{e7} is obtained for the function $\omega_2(z)$ given in Theorem \ref{th-1},
\[ \omega_2(z) = \int_0^z \frac{|c_1|+z}{1+|c_1|z} \,dz = |c_1|z+\frac12(1-|c_1|^2)z^2-\frac13|c_1|(1-|c_1|^2)z^3+\cdots.\]
\end{proof}

\medskip

\begin{rem}
From \eqref{e7} we have taht for every $0\le|c_1|\le1$,
\[ |c_1c_3-c_2^2| \le \frac{1}{12}\left( 3-2|c_1|^2-|c_1|^4  \right) \le\frac14.\]
\end{rem}

\medskip

\begin{rem}
As it is shown in \cite{Mariae}, for the class $\U$ of functions $f(z)=z+a_2z^2+a_3z^3+\cdots$ defined by the condition
\[ \left| \left( \frac{z}{f(z)} \right)^2 f'(z)-1\right| <1,\qquad z\in\D,\]
we have
\begin{equation}\label{e8}
\frac{z}{f(z)} = 1-a_2 z-z\omega(z),
\end{equation}
where $\omega\in\mathcal{B}'_0$ and $\omega(z)=c_1z+c_2z^2+\cdots$. From \eqref{e8} we can express the coefficients $a_3$, $a_4$, and $a_5$, of the function $f$, depending on $a_2$, $c_1$, $c_2,$, $c_3$,\ldots. After some calculations we receive
\[ |H_3(1)(f)| = \left|c_1c_3-c_2^2\right|\le\frac14, \]
where $H_3(1)(f)$ is the Hankel determinant of third order (see \cite{book}) and that result is the best possible. This property was the inspiration to study the class $\mathcal{B}'_0$ as a continuation of the study of the class $\mathcal{B}_0$ in \cite{zaprawa}.
\end{rem}

\medskip

Similarly as in Theorem \ref{th-2} we get

\begin{thm}\label{th-4}
If $\omega\in \mathcal{B}'_0$ is of form \eqref{e1} and $\mu\in\CC$, then
\begin{equation}\label{e6}
  |c_1c_3-\mu c_2^2|\le
  \left\{
  \begin{array}{cc}
    \frac{1}{3}|c_1|(1-|c_1|^2), & |\mu| \le\frac43 \frac{|c_1|}{1+|c_1|} \\[2mm]
    \frac{1}{12}\left[ 3|\mu|+2(2-3|\mu|)|c_1|^2-(4-3|\mu|)|c_1|^4 \right], & |\mu| \ge\frac43 \frac{|c_1|}{1+|c_1|}
  \end{array}
  \right..
\end{equation}
\end{thm}

\medskip

Finally, for the modulus $ |c_4-c_2^2|$ we have

\medskip

\begin{thm}\label{th-5}
If $\omega\in \mathcal{B}'_0$ is of form \eqref{e1}, then
\begin{equation}\label{e10}
  |c_4-c_2^2|\le     \frac{1}{4}(1-|c_1|^2)
\end{equation}
and the estimate is sharp as the function
\[ \omega(z) = \int_0^z \frac{|c_1|+z^3}{1+|c_1|z^3} \,dz = |c_1|z+\frac14(1-|c_1|^2)z^4-\frac16|c_1|(1-|c_1|^2)z^6+\cdots.\]
shows.
\end{thm}

\begin{proof}
Using Lemma \ref{lem-B}, we easily get
\[ |c_4-c_2^2|\le |c_4|+|c_2|^2 \le     \frac{1}{4}(1-|c_1|^2-4|c_2|^2) + |c_2|^2  =  \frac{1}{4}(1-|c_1|^2) .\]
\end{proof}

\end{document}